\documentclass[11pt,a4paper]{article}
\usepackage{amsmath,amsthm,amssymb}
\usepackage{amscd}
\usepackage{cases}
\usepackage[german,french,english]{babel}
\usepackage{setspace}
\makeatletter
 
  \@addtoreset{equation}{section}
\makeatother
\newtheorem{thm}{Theorem}[section]
\newtheorem{lem}[thm]{Lemma}
\newtheorem{prop}[thm]{Proposition}
\newtheorem{cor}[thm]{Corollary}

\begin{document}
\title{The extension of some $D(4)$-pairs}
\author{Alan Filipin}

\maketitle
\begin{abstract}
In this paper we illustrate the use of the results from \cite{MC-BF} proving that $D(4)$-triple $\{a,b,c\}$ with $a<b<a+57\sqrt{a}$ has a unique extension to a quadruple with a larger element. This furthermore implies that $D(4)$-pair $\{a,b\}$ cannot be extended to a quintuple if $a<b<a+57\sqrt{a}$.
\end{abstract}
\bigskip{\textrm{2010} \textit{Mathematics Subject Classification}: Primary 11D09; Secondary 11J68}
\section{Introduction}\label{sec:intr}
Let $n$ be a nonzero integer. The set $\{a_1,a_2,\dots,a_m\}$ of $m$ positive integers is called a Diophantine $m$-tuple with the property $D(n)$ or simply $D(n)$-$m$-tuple if $a_ia_j+1$ is a perfect square for all $i,j$ with $1 \leq i<j \leq m$.

The problem of finding such sets has a long and rich history. Reader should visit the web-page \cite{duje} to find all problems on this topic. More precisely, that includes problems which have been solved over the years, open problems, conjectures, methods used in solving those problems with their improvement in recent years and all the references.

In the case $n=4$ there is a conjecture \cite{DR} that there does not exist a $D(4)$-quintuple. Actually, the conjecture is that if $\{a,b,c,d\}$ is a $D(4)$-quadruple with $a<b<c<d$, then $$d=d_+=a+b+c+\frac{1}{2}\left(abc+rst\right),$$ where $r,s$ and $t$ are positive integers defined by $r^2=ab+4,$ $s^2=ac+4$ and $t^2=bc+4$. From now on, $r$, $s$ and $t$ will always be defined in such a way. If $\{a,b,c,d\}$ is a $D(4)$-quadruple with $a<b<c<d$, and $d=d_+$, we call such quadruple a regular one. The conjecture also implies that all $D(4)$-quadruples are regular. We also define $$d_-=a+b+c+\frac{1}{2}\left(abc-rst\right).$$ If $d_-\neq0$, the set $\{a,b,d_-,c\}$ is a regular $D(4)$-quadruple with $d_-<c$.

In recent years the author \cite{JNT} proved that there does not exist a $D(4)$-sextuple and that there are only finitely many quintuples. He furthermore \cite{acta} proved that irregular quadruple cannot be extended to a quintuple with a larger element and, very recently, together with Ba\'ci\' c \cite{HAZU} he proved that there are at most $7\cdot 10^{36}$ $D(4)$-quintuples which is the best known bound for now.

In this paper we show the important use of the results from \cite{MC-BF} which is the main intention of this paper. If $\{a,b\}$ is a $D(4)$-pair with $a<b<a+57\sqrt{a}$, we know firstly from \cite[Lemma 3]{MC-BF} that we have to consider only the case when $b>10^4$, which implies $a<b<2a$. Then, from \cite[Lemma 1]{MC-BF} we know how to generate all possible $c$'s in $D(4)$-triple $\{a,b,c\}$. Precisely, we have $c=c_{\nu}^{\pm}$ where $$c_0^+=0,\,c_1^+=a+b+2r,\,c_{\nu+2}^+=(ab+2)c_{\nu+1}^+-c_{\nu}^++2(a+b),$$ $$c_0^-=0,\,c_1^-=a+b-2r,\,c_{\nu+2}^-=(ab+2)c_{\nu+1}^--c_{\nu}^-+2(a+b).$$
And finally, in \cite[Theorem 1]{MC-BF} we proved that if $\{a,b,c\}$ is $D(4)$-triple with $a<b<2a$ and $\{a,b,c,d\}$ $D(4)$-quadruple with $d>d^+$ and if there does not exist a $D(4)$-quadruple $\{a,b,c',c\}$ with $0<c'<d_-$, then $c<b^6$.

The assumption that there does not exist a $D(4)$-quadruple $\{a,b,c',c\}$ with $0<c<d_-$ is not as restrictive as it might seem to be, because we know exactly which $c$'s we have to consider and how those elements are ordered, so we will later check that this assumption is satisfied starting with the smallest possible $c$.

Our main results is the following Theorem.
\begin{thm}\label{thm:main1}
If $\{a,b,c,d\}$ is a $D(4)$-quadruple with $a<b < a+57\sqrt{a}$ and $c<d$, then $d=d_+$.
\end{thm}

It obviously implies the following Corollary.

\begin{cor} The $D(4)$-pair $\{a,b\}$ with $a<b<a+57\sqrt{a}$ cannot be extended to a quintuple.
\end{cor}

In the proof of Theorem we use standard methods, i.e. transforming the problem of the extension of $D(4)$-triple into the finding of the intersection of binary recurrence sequences which is then solved using congruence method together with Baker's theory of linear forms in logarithms, and in the end we do Baker-Davenport reduction method based on continued fractions. Let us mention that analogous result in the case of $D(1)$-$m$-tuples was very recently proved in \cite{FFT}. In the end we discuss if this method can work in other cases too. Also, the important fact is that we use the linear forms in two logarithms instead in three which gives us much better bounds. However, for this method to be used, it is necessary that the elements $a$ and $b$ are near to each other other and, therefore, it cannot be used in general. That method was already used in several papers and was firstly introduced in \cite{HT}.

\section{Preliminaries}
Let $\{a,b,c\}$ be a $D(4)$-triple with $a<b<57\sqrt{a}$. As we mentioned in the last Section, we have to consider only what is happening with $c<b^6$. It is easy to check that $c_4^+>b^6$, so we have to consider only $c=c_1^-,c_1^+,c_2^-,c_2^+,c_3^-,c_3^+,c_4^-$. We can also from now on assume that there does not exist a $D(4)$-quadruple $\{a,b,c',c\}$ with $0<c'<d_-$ because if we prove that we can have only the known extensions to a quadruple (with $d=d_{\pm}$), for $c\leq c_4^-$, we prove that assumption is valid for all $c$'s combining it with the proof of \cite[Theorem 1]{MC-BF}.

Let now $\{a,b,c,d\}$ be a $D(4)$-quadruple with $\max\{a,b,c\}<d$. Then, there exist positive integers $x,y,z$ such that $ad+4=x^2$, $bd+4=y^2$, $cd+4=z^2$.
Eliminating $d$ from these relations, we get
\begin{align}
ay^2-bx^2&=4(a-b),\label{D:yx}\\
az^2-cx^2&=4(a-c),\label{D:zx}\\
bz^2-cy^2&=4(b-c).\label{D:zy}
\end{align}
Since $a<b<a+57\sqrt{a}$, \cite[Lemma 1]{MC-BF} implies that the positive solutions of the equation (\ref{D:yx}) are given by
\begin{align}\label{sol:yx}
y\sqrt{a}+x\sqrt{b}=(\pm 2\sqrt{a}+2\sqrt{b})\left(\frac{r+\sqrt{ab}}{2}\right)^l
\end{align}
with some non-negative integer $l$.
So we have $x=p_l$, $y=V_l$, where
\begin{align}
p_0&=2,\enskip p_1=r \pm a,\enskip p_{l+2}=rp_{l+1}-p_l,\label{seq:x1}\\
V_0&=\pm2,\,V_1=b \pm r,\,V_{l+2}=rV_{l+1}-V_l.\label{seq:y1}
\end{align}
Furthermore, by \cite[Lemma 1]{rocky} all positive solutions of the equations (\ref{D:zx}) and (\ref{D:zy}) are given by
\begin{align}
z\sqrt{a}+x\sqrt{c}&=(z_0\sqrt{a}+x_0\sqrt{c})\left(\frac{s+\sqrt{ac}}{2}\right)^m,\label{sol:zx}\\
z\sqrt{b}+y\sqrt{c}&=(z_1\sqrt{b}+y_1\sqrt{c})\left(\frac{t+\sqrt{bc}}{2}\right)^n,\label{sol:zy}
\end{align}
where $m,n$ are non-negative integers, and $(z_0,x_0)$, $(z_1,y_1)$ are fundamental solutions of (\ref{D:zx}) and (\ref{D:zy}) respectively. Moreover, we have the following estimates for the fundamental solutions
\begin{align}
2 \leq \,&x_0 <\sqrt{s+2},\quad
2 \leq \,|z_0| <\sqrt{\frac{c\sqrt{c}}{\sqrt{a}}},\label{in:xz}\\
2 \leq \,&y_1 <\sqrt{t+2},\quad
2 \leq \,|z_1| <\sqrt{\frac{c\sqrt{c}}{\sqrt{b}}},\label{in:yz}
\end{align}
if $c>b$.
In the case where $c<b$, which in our case also implies $c<a$, positive solutions of the equations \eqref{D:zx} and (\ref{D:zy}) are given by \eqref{sol:zx} and \eqref{sol:zy} with the estimates
\begin{align}
2 \leq z_0<\sqrt{s+2},\quad 2 \leq |x_0|<\sqrt{\frac{a\sqrt{a}}{\sqrt{c}}},\label{in:zx}\\
2 \leq z_1<\sqrt{t+2},\quad 2 \leq |y_1|<\sqrt{\frac{b\sqrt{b}}{\sqrt{c}}}.\label{in:zy}
\end{align}
However, in all cases, we have $z=v_m=w_n$, where
\begin{align}
v_0&=z_0,\enskip v_1=\frac{1}{2}\left(sz_0+cx_0\right),\enskip v_{m+2}=sv_{m+1}-v_m,\notag\\
w_0&=z_1,\enskip w_1=\frac{1}{2}\left(tz_1+cy_1\right),\enskip v_{n+2}=tw_{n+1}-w_n.\label{seq:w}
\end{align}
If necessary, we will also consider $x=q_m$ and $y=W_n$, where
\begin{align}
q_0&=x_0,\enskip q_1=\frac{1}{2}\left(sx_0+az_0\right),\enskip q_{m+2}=sq_{m+1}-q_m,\label{seq:x2}\\
W_0&=y_1,\enskip W_1=\frac{1}{2}\left(ty_1+bz_1\right),\enskip W_{n+2}=tW_{n+1}-W_n.\label{seq:y2}
\end{align}
%
In the next Lemma we will compute the initial values of the given sequences.
\begin{lem}\label{lem:fs}
\begin{itemize}
\item[{\rm (1)}]
If $c=c_1^{-}$, we have $v_{2m+1}\neq w_{2n}$, $v_{2m}\neq w_{2n+1}$ and $v_{2m+1}\neq w_{2n+1}$.
Moreover, if $v_{2m}=w_{2n}$, then $z_0=z_1=2$.
\item[{\rm (2)}] If $c=c_1^{+}$, we have $v_{2m+1}\neq w_{2n}$, $v_{2m}\neq w_{2n+1}$ and $v_{2m+1}\neq w_{2n+1}$.
Moreover, if $v_{2m}=w_{2n}$, then $z_0=z_1$ and $|z_0|=2$.
\item[{\rm (3)}] If $c \in \{c_2^-,c_2^+,c_3^-,c_3^+,c_4^-\}$, then it is enough to consider the following:
\begin{itemize}
\item[{\rm (i)}] If $v_{2m}=w_{2n}$, then $z_0=z_1$ and $|z_0|=2$.
\item[{\rm (ii)}] If $v_{2m+1}=w_{2n+1}$, then $|z_0|=t$ and $|z_1|=s$ with $z_0z_1>0$.
\end{itemize}
\end{itemize}
\end{lem}
\begin{proof}
The cases $(1)$ and $(2)$ can be proven in the same way as discussion after \cite[Lemma 1]{glasnik} using the estimates (\ref{in:xz}), (\ref{in:yz}), (\ref{in:zx}) and (\ref{in:zy}) together with \cite[Lemma 1]{rocky}. It implies right away the statement $(2)$, while in $(1)$ it implies $y_1=\pm 2$. Then, we conclude $z_1=2$. To get $z_0=2$ we have to prove that $x_0=2$ which follows from $x=p_l=q_m$ considering congruences modulo $a$ and noticing that it would give us $x_0=2$ or $r-a=2$. The latter implies $b=a+4$, the case which was solved completely in \cite{JCNT} considering the extension of $D(4)$-pairs $\{k-2,k+2\}$.

The case $(3)$ can be proven the same way as \cite[Lemma 2]{MC-BF}.
\end{proof}

The following Lemma can be proven easily.
\begin{lem}\label{lem:m>n}
If $v_{m'}=w_{n'}$ has a solution with $n' \geq 3$, then $m'>n'$.
\end{lem}

Note that last Lemma implies $m'-n' \geq 2$.
\section{Linear forms in logarithms}
Now using standard method in linear forms in logarithms we can prove that if $c \in \{c_1^+,c_2^-,c_2^+,c_3^-,c_3^+,c_4^-\}$ and $v_{m'}=w_{n'}$, then
\begin{align}\label{lf1}
0<\Lambda<\alpha_1^{2-2m'},
\end{align}
where
\begin{align*}
\Lambda&=m' \log \alpha_1-n' \log \alpha_2+\log \mu
\end{align*}
and
\begin{align*}
\alpha_1&=\frac{s+\sqrt{ac}}{2},\enskip \alpha_2=\frac{t+\sqrt{bc}}{2},\enskip
\mu=\frac{\sqrt{b}(x_0\sqrt{c}+z_0\sqrt{a})}{\sqrt{a}(y_1\sqrt{c}+z_1\sqrt{b})}.
\end{align*}
Moreover, if $c=c_1^-$ and $v_{m'}=w_{n'}$, then
\begin{align}\label{lf2}
0<\Lambda<\alpha_2^{1-2n'}.
\end{align}
%
Remember that we consider only the case $a<b<a+57\sqrt{a}$ and $c\leq c_4^-$.
\begin{lem}\label{lem:lambda<}
\begin{itemize}
\item[{\rm (1)}] If $c=c_1^-<a-4$.
and $v_{m'}=w_{n'}$ has a solution with $n'\geq 3$, then
\[
m'\log \alpha_1-(n'+0.001)\log \alpha_2<0.
\]
\item[{\rm (2)}] If $c \in \{c_1^+,c_2^-,c_2^+,c_3^-,c_3^+,c_4^-\}$ and $v_{m'}=w_{n'}$ has a solution with $n' \geq 3$, then
\[
(m'-0.001)\log \alpha_1-n'\log \alpha_2<0.
\]
\end{itemize}
\end{lem}
\begin{proof}
(1) Since $b>10^4$, it implies $a\geq5700$. Then, it is easy to see that $\mu>0.99$ and $\alpha_2^5\log \alpha_2>10^4$. Then, we conclude from \eqref{lf2} that
\[
\Lambda<\alpha_2^{-5}<0.001\log \alpha_2+\log \mu,
\]
from which the statement of Lemma follows. \par
$(2)$ can be proven in the same way.
\end{proof}

\begin{lem}\label{lem:lb-1}
Let $\nu'=m'-n'$.
\begin{itemize}
\item[{\rm (1)}] If $c=c_1^-$ and $v_{m'}=w_{n'}$ has a solution with $n' \geq 2$,
then
\[
n'>2/57(\nu'-0.001)\sqrt{a}\log \alpha_1-0.001.
\]
\item[{\rm (2)}] If $c \in\{c_1^+,c_2^-,c_2^+,c_3^-,c_3^+,c_4^-\}$ and $v_{m'}=w_{n'}$ has a solution with $n' \geq 3$, then
\[
n'>2/57(\nu'-0.001)\sqrt{a}\log \alpha_1.
\]
\end{itemize}
\end{lem}
\begin{proof}
(1) By Lemma \ref{lem:lambda<}, we have
\begin{align*}
\frac{\nu'-0.001}{n'+0.001}&=\frac{m'}{n'+0.001}-1<\frac{\log \alpha_2}{\log \alpha_1}-1\\
                           &<\frac{\alpha_2-\alpha_1}{\alpha_1\log \alpha_1}<\frac{b-a}{\sqrt{a}(\sqrt{a}+\sqrt{b})\log \alpha_1}\\
                           &<\frac{57}{2\sqrt{a}\log \alpha_1},
\end{align*}
which proves the inequality. \par
$(2)$ can be proven similarly using Lemma \ref{lem:lambda<}.
\end{proof}

We are now ready to give an upper bound for $a$. To do that we use the important result from \cite{Mig}. For any nonzero algebraic number $\gamma$ of degree $d'$ over $\mathbb{Q}$,
\[
h(\gamma)=\frac{1}{d'}\left(\log|a_0|+\sum_{j=1}^{d'}\log \max \left\{1,\left|\gamma^{(j)}\right|\right\}\right)
\]
denotes its absolute logarithmic height,
where $a_0$ is the leading coefficient of the minimal polynomial of $\alpha$ over $\mathbb{Z}$ and $\gamma^{(j)}$ are the complex conjugates of $\gamma$.
\begin{thm}\textup{(\cite[Corollary of Theorem 2]{Mig})}\label{thm:Mign}
Let $\gamma_1$ and $\gamma_2$ be multiplicative independent positive real numbers.
For positive integers $b_1$ and $b_2$, define $\Lambda=b_1 \log \gamma_1-b_2 \gamma_2$.
Put $D=[\mathbb{Q}(\gamma_1,\gamma_2):\mathbb{Q}]/[\mathbb{R}(\gamma_1,\gamma_2):\mathbb{R}]$.
Let $\rho$, $\kappa$ and $a_i$ $($$i \in \{1,2\}$$)$ be positive real numbers with $\rho \geq 4$, $\kappa=\log \rho$,
\[
a_i \geq \max\{1,(\rho-1)\log|\gamma_i|+2Dh(\gamma_i)\}
\]
and
\[
a_1a_2 \geq \max \left\{20,4\kappa^2\right\}.
\]
Suppose that $h$ is a real number with
\[
h \geq \max\left\{3.5,1.5\kappa,D \left(\log \left(\frac{b_1}{a_2}+\frac{b_2}{a_1}\right)+\log \kappa+1.377 \right)+0.023 \right\},
\]
and put $\chi=h/\kappa$, $v=4\chi+4+1/\chi$.
Then, we have
\begin{align*}
\log|\Lambda| \geq -(C_0+0.06)(\kappa+h)^2a_1a_2,
\end{align*}
where
\[
C_0=\frac{1}{\kappa^3}\left\{\left(2+\frac{1}{2\chi(\chi+1)}\right)
\left(\frac13+\sqrt{\frac19+\frac{4\kappa}{3v}\left(\frac{1}{a_1}+\frac{1}{a_2}\right)+\frac{32\sqrt{2}(1+\chi)^{3/2}}{3v^2\sqrt{a_1a_2}}}\right)\right\}^2.
\]
\end{thm}
\begin{prop}\label{prop:1M}
Let $\{a,b\}$ be a $D(4)$-pair with $a<b<a+57\sqrt{a}$.
Let $c \in \{c_1^-,c_1^+,c_2^-,c_2^+,c_3^-,c_3^+,c_4^-\}$.
Suppose that $\{a,b,c,d\}$ is an irregular $D(4)$-quadruple for some $d$. Then $a<6.55\cdot 10^{11}$.
\end{prop}
\begin{proof}
We apply Theorem \ref{thm:Mign} to our $\Lambda$ by rewriting the linear form
\[
\Lambda=\log(\alpha_1^{\nu'}\mu)-n'\log \left(\frac{\alpha_2}{\alpha_1}\right).
\]
So in our case we have
\[
D=4,\enskip b_1=1,\enskip b_2=n',\enskip \gamma_1=\alpha^{\nu'}\mu,\enskip \gamma_2=\alpha_2/\alpha_1.
\]
Here we will give only the proof for $c=c_1^-$ because other cases can be proven similarly. Also, from now on, let us assume $a>10^{10}$.

Firstly, it is easy to see
$h(\gamma_2)=h(\alpha_2/\alpha_1)\leq h(\alpha_2)+h(\alpha_1)
           \leq \log\alpha_2$.

In order to estimate $h(\gamma_1)$, we have to bound $h(\mu)$. In general, the leading coefficient of the minimal polynomial of $\mu$ over $\mathbb{Z}$ divides $16a^2(b-c)^2$. If $c=c_1^-$, then since the absolute values of conjugates of $\mu$ greater than 1 are
\[
\frac{\sqrt{b}(\sqrt{a}+\sqrt{c})}{\sqrt{a}(\sqrt{b}+\sqrt{c})},\enskip
\frac{\sqrt{b}(\sqrt{a}+\sqrt{c})}{\sqrt{a}(\sqrt{b}-\sqrt{c})},
\]
we have
\begin{align}
h(\mu)&\leq \frac14 \log\left\{a^2(b-c)^2\cdot \frac{b(\sqrt{a}+\sqrt{c})^2}{a(b-c)}\right\}\notag\\
      &<\frac14 \log(1.001b^4c^4)<2.001\log \alpha_2.\label{h1':1-}
\end{align}

Then we have
\begin{align*}
h(\gamma_1)=h(\alpha_1^{\nu'}\mu)\leq \nu'h(\alpha_1)+h(\mu)<\left(0.5\nu'+2.001 \right)\log \alpha_2.
\end{align*}

Now we can take $\rho=5$ and $a_2=8.348\log\alpha_2$. Moreover, if $c=c_1^-$, from $a>10^{10}$ we get $c <400$, which implies
\begin{align*}
\mu &\leq \frac{1-\sqrt{c/a}}{1-\sqrt{c/b}}<1.001,\\
\alpha_2 &>100028.
\end{align*}
Then, we have
\[
\log \mu+4.002\log \alpha_2<4.003\log \alpha_2,
\]
which enables us to take
\[
a_1=8(\nu'+2.002)\log \alpha_2.
\]
%
From Lemma \ref{lem:lb-1} we now have
\begin{align*}
\frac{b_1}{a_2}&=\frac{1}{8.348\log \alpha_2}<\frac{n'(\nu'+2.002)}{2/57(\nu'-0.001)\sqrt{a}\log \alpha_1}\cdot \frac{1}{8(\nu'+2.002)\log \alpha_2}\\
               &<0.001\cdot \frac{b_2}{a_1}.
\end{align*}

Then we may take
\[
h=4\log \left(\frac{n'}{(\nu'+2.002)\log \alpha_2}\right)-2.306.
\]
If $h \geq 35$, then $C_0<0.2411$.
It follows from (\ref{lf2}) and Theorem \ref{thm:Mign} that
\[
\frac{n'}{(\nu'+2.002)\log \alpha_2}<10.055\cdot\left(4\log \left(\frac{n'}{(\nu'+2.002)\log \alpha_2}\right)-0.696\right)^2,
\]
which implies that
\[
\frac{n'}{(\nu'+2.002)\log \alpha_2}<14170.
\]
If $h<35$, then it yields
\[
\frac{n'}{(\nu'+2.002)\log \alpha_2}<11231<14170.
\]
It follows now from Lemma \ref{lem:lb-1} that
\[
\frac{2/57(\nu'-0.001)\sqrt{a}\log \alpha_1-0.001}{(\nu'+2.002)\log \alpha_2}<14170.
\]
Since $\nu' \geq 2$ and
\[
\frac{\log \alpha_1}{\log \alpha_2}>0.999,
\]
we obtain $a<6.55\cdot 10^{11}$. \par

\end{proof}

\section{Concluding remarks}

To finish the proof of the main Theorem \ref{thm:main1}, we have to check what is happening with small values of $a$, i.e. $a<6.55\cdot10^{11}$. We do that using the already pretty standard Baker-Davenport reduction method. However, since the bound for $a$ is very large here, we have to use that $a<b<a+57\sqrt{a}$.

It implies that $a<r<a+57/2\sqrt{a}$ or $(r-a)^2<813a$. Moreover, since $(r-a)^2\equiv 4\pmod{a}$, we get $(r-a)^2=4,\,a+4,2a+4,\ldots,812a+4$. So it will give us 3691 parametric families of $D(4)$-pairs to consider their extension. The families we get are:
$$\{k^2-4,k^2+2k-3\},$$ $$\{2k^2-2,2k^2+4k\},$$ $$\{3k^2-2k-1,3k^2+4k\},$$ $$\{3k^2+2k-1,3k^2+8k+4\},$$ $$.$$ $$.$$ $$.$$ $$\{812k^2-4k,812k^2+1620k+808\},$$$$\{812k^2+4k,812k^2+1628k+816\},$$ $$\{812k^2-228k+16,812k^2+1396k+600\},$$ $$\{812k^2+228k+16,812k^2+1852k+1056\},$$ $$\{812k^2-584k+105,812k^2+1040k+333\},$$ $$\{812k^2+584k+105,812k^2+2208k+1501\},$$ $$\{812k^2-808k+201,812k^2+816k+205\},$$ $$\{812k^2+808k+201,812k^2+2432k+1821\}.$$

For all those families we have $a$ and $b$ fixed, we get an upper bound for $k$ and we know for which $c$'s we have to do the reduction. We implement that in Mathematica and get that only extensions to $D(4)$-triple $\{a,b,c\}$ are given by $d=0$ and $d=d_{\pm}$ which finishes the proof of Theorem \ref{thm:main1}.\\

Finally, here are some observations. This method would also work if, for example $a<b<178\sqrt{a}<5a$, because we also know how $c$'s which extend that pair are given. However, that way we would get a larger bound on $a$ but it would also leave us with many more parametric families to consider and it would take years to make the reduction using today's computers.

The same method would, at least in theory, also work with $a<b<a+a^N$ if $N<1$. But in general that way we would also get much larger bound for $a$ and we would not be able to find parametric families of $D(4)$-pairs which we have to consider. Therefore, to solve the problem with $b<5a$ or something similar the new ideas or approaches are needed, at least for proving strong quintuple conjecture, i.e. that $D(4)$-triple has a unique extension to a quadruple with a larger element.

\vspace{5mm}
{\bf Acknowledgements}
The author is supported by Croatian Science Foundation grant number 6422.


Faculty of Civil Engineering, University of Zagreb, Fra Andrije Ka\v{c}i\'{c}a-Mio\v{s}i\'{c}a 26, 10000 Zagreb, Croatia \\
E-mail: filipin@master.grad.hr \\[6pt]
\end{document}